\documentclass{amsart}
\usepackage{amssymb}
\usepackage{amsmath}
\usepackage{amsfonts}

\newtheorem{theorem}{Theorem}
\theoremstyle{plain}

\newtheorem{corollary}{Corollary}

\newtheorem{lemma}{Lemma}

\newtheorem{proposition}{Proposition}

\numberwithin{equation}{section}
\input{tcilatex}

\begin{document}
\title{NEW INEQUALITIES FOR $n-$ TIME DIFFERNTIABLE FUNCTIONS }
\author{M. EM\.{I}N \"{O}ZDEM\.{I}R$^{\bigstar }$}
\address{$^{\bigstar }$ATAT\"{U}RK UNIVERSITY, K. K. EDUCATION FACULTY,
DEPARTMENT OF MATHEMATICS, 25240, CAMPUS, ERZURUM, TURKEY}
\email{emos@atauni.edu.tr}
\author{\c{C}ET\.{I}N YILDIZ$^{\bigstar ,\spadesuit }$}
\email{cetin@atauni.edu.tr}
\thanks{$^{\spadesuit }$Corresponding Author.}

\begin{abstract}
In this paper, we obtain several inequalities of Ostrowski type that the
absolute values of n-time differntiable functions are convex.
\end{abstract}

\maketitle

\section{INTRODUCTION}

In 1938 Ostrowski \cite{ostr} obtained a bound for the absolute value of the
difference of a function to its average over a finite interval. The theorem
is as follows.

\begin{theorem}
Let $f:[a,b]\rightarrow 
%TCIMACRO{\U{211d} }%
%BeginExpansion
\mathbb{R}
%EndExpansion
$ be a differentiable mapping on $[a,b]$ and let $\left\vert f^{\prime
}(t)\right\vert \leq M$ for all $t\in (a,b)$, then the following bound is
valid%
\begin{equation}
\left\vert f(x)-\frac{1}{b-a}\int_{a}^{b}f(t)dt\right\vert \leq (b-a)M\left[ 
\frac{1}{4}+\frac{\left( x-\frac{a+b}{2}\right) ^{2}}{(b-a)^{2}}\right]
\label{1}
\end{equation}%
for all $x\in \lbrack a,b].$ The constant $\frac{1}{4}$ is sharp in the
sence that it can not be replaced by a smaller one.
\end{theorem}

For applications of Ostrowski's inequality to some special means and some
numerical quadrature rules, we refer the reader to the recent paper \cite{c}
by S.S. Dragomir and S. Wang who used integration by parts from $%
\int_{a}^{b}p(x,t)f^{\prime }(t)dt$ to prove Ostrowski's inequality (\ref{1}%
) where $p(x,t)$ is a peano kernel given by%
\begin{equation*}
p(x,t)=\left\{ 
\begin{array}{cc}
t-a, & t\in \lbrack a,x] \\ 
t-b, & t\in (x,b].%
\end{array}%
\right.
\end{equation*}

In \cite{d}, also A. Sofo and S.S Dragomir\textit{\ }extended the result (%
\ref{1}) in the $Lp$ norm.

Dragomir (\cite{drag2}-\cite{b}) further extended the result (\ref{1}) to
incorporate mappings of bounded variation, Lipschitzian and monotonic
mappings.

Cerone \textit{et al.} \cite{ceron} as well as Dedi\'{c} \textit{et al.} 
\cite{dedic} and Pearce \textit{et al.} \cite{pearc} further extended the
result (\ref{1}) by considering $n$-times differentiable mappings on an
interior point $x\in \lbrack a,b]$. Furthermore, for recent results and
generalizations concerning Ostrowski's inequality see \cite{al}, \cite{kav}-%
\cite{e2}, \cite{set} and \cite{set2}.

In \cite{ceron}, Cerone, Dragomir and Roumeliotis proved the following
results:

\begin{lemma}
\label{cet}Let $f:[a,b]\rightarrow 
%TCIMACRO{\U{211d} }%
%BeginExpansion
\mathbb{R}
%EndExpansion
$ be a mapping such that $f^{(n-1)}$ is absolutely continuous on $[a,b].$
Then for all $x\in \lbrack a,b]$ we have the identity:%
\begin{eqnarray*}
\int_{a}^{b}f(t)dt &=&\overset{n-1}{\underset{k=0}{\sum }}\left[ \frac{%
(b-x)^{k+1}+(-1)^{k}(x-a)^{k+1}}{(k+1)!}\right] f^{(k)}(x) \\
&&+(-1)^{n}\int_{a}^{b}K_{n}(x,t)f^{(n)}(t)dt
\end{eqnarray*}%
where the kernel $K_{n}:[a,b]^{2}\rightarrow 
%TCIMACRO{\U{211d} }%
%BeginExpansion
\mathbb{R}
%EndExpansion
$ is given by%
\begin{equation*}
K_{n}(x,t)=\left\{ 
\begin{array}{cc}
\frac{(t-a)^{n}}{n!} & if\text{ }t\in \lbrack a,x] \\ 
&  \\ 
\frac{(t-b)^{n}}{n!} & if\text{ }t\in (x,b],%
\end{array}%
\right.
\end{equation*}%
$x\in \lbrack a,b]$ and $n$ natural number, $n\geq 1.$
\end{lemma}

\begin{corollary}
With the above assumptions, we have the representation:%
\begin{eqnarray*}
\int_{a}^{b}f(t)dt &=&\overset{n-1}{\underset{k=0}{\sum }}\left[ \frac{%
1+(-1)^{k}}{(k+1)!}\right] \frac{(b-a)^{k+1}}{2^{k+1}}f^{(k)}\left( \frac{a+b%
}{2}\right) \\
&&+(-1)^{n}\int_{a}^{b}M_{n}(t)f^{(n)}(t)dt
\end{eqnarray*}%
where 
\begin{equation*}
M_{n}(t)=\left\{ 
\begin{array}{cc}
\frac{(t-a)^{n}}{n!} & if\text{ }t\in \left[ a,\frac{a+b}{2}\right] \\ 
&  \\ 
\frac{(t-b)^{n}}{n!} & if\text{ }t\in \left( \frac{a+b}{2},b\right] .%
\end{array}%
\right.
\end{equation*}
\end{corollary}

\begin{corollary}
With the above assumptions, we have the representation:%
\begin{eqnarray*}
\int_{a}^{b}f(t)dt &=&\overset{n-1}{\underset{k=0}{\sum }}\frac{(b-a)^{k+1}}{%
(k+1)!}\left[ \frac{f^{(k)}\left( a\right) +(-1)^{k}f^{(k)}\left( b\right) }{%
2}\right] \\
&&+\int_{a}^{b}T_{n}(t)f^{(n)}(t)dt
\end{eqnarray*}%
where 
\begin{equation*}
T_{n}(t)=\frac{1}{n!}\left[ \frac{(b-t)^{n}+(-1)^{n}(t-a)^{n}}{2}\right] ,
\end{equation*}%
$t\in \lbrack a,b].$
\end{corollary}

In this paper, by using the some classical integral inequalities, H\"{o}lder
and Power-Mean integral inequality, we establish some new inequalities for
functions whose $n-th$ derivatives in absolute value are convex functions.
Our established results generalize some of those results proved in recent
papers for functions whose derivatives in absolute value are convex
functions.

\section{MAIN RESULTS}

\begin{theorem}
\label{tin}For $n\geq 1,$ let $f:[a,b]\rightarrow 
%TCIMACRO{\U{211d} }%
%BeginExpansion
\mathbb{R}
%EndExpansion
$ be $n-$time differentiable mapping and $a<b.$ If $f^{(n)}\in L[a,b]$ and $%
\left\vert f^{(n)}\right\vert $ is convex on $[a,b],$ then 
\begin{eqnarray}
&&\left\vert \int_{a}^{b}f(t)dt-\overset{n-1}{\underset{k=0}{\sum }}\left[ 
\frac{(b-x)^{k+1}+(-1)^{k}(x-a)^{k+1}}{(k+1)!}\right] f^{(k)}(x)\right\vert
\label{3} \\
&\leq &\frac{1}{n!(b-a)}\left\{ \left\vert f^{(n)}(a)\right\vert \left[ 
\frac{(x-a)^{n+1}\left[ (n+2)(b-x)+(x-a)\right] }{(n+1)(n+2)}+\frac{%
(b-x)^{n+2}}{(n+2)}\right] \right.  \notag \\
&&\text{ \ \ \ \ \ \ \ \ \ \ }\left. +\left\vert f^{(n)}(b)\right\vert \left[
\frac{(b-x)^{n+1}\left[ (n+2)(x-a)+(b-x)\right] }{(n+1)(n+2)}+\frac{%
(x-a)^{n+2}}{(n+2)}\right] \right\} .  \notag
\end{eqnarray}
\end{theorem}

\begin{proof}
From Lemma \ref{cet} and using the properties of modulus, we write%
\begin{eqnarray}
&&\left\vert \int_{a}^{b}f(t)dt-\overset{n-1}{\underset{k=0}{\sum }}\left[ 
\frac{(b-x)^{k+1}+(-1)^{k}(x-a)^{k+1}}{(k+1)!}\right] f^{(k)}(x)\right\vert
\label{2} \\
&\leq &\left\vert \int_{a}^{b}K_{n}(x,t)f^{(n)}(t)dt\right\vert  \notag \\
&=&\int_{a}^{x}\frac{(t-a)^{n}}{n!}\left\vert f^{(n)}(t)\right\vert
dt+\int_{x}^{b}\frac{(b-t)^{n}}{n!}\left\vert f^{(n)}(t)\right\vert dt 
\notag \\
&=&\int_{a}^{x}\frac{(t-a)^{n}}{n!}\left\vert f^{(n)}\left( \frac{b-t}{b-a}a+%
\frac{t-a}{b-a}b\right) \right\vert dt  \notag \\
&&+\int_{x}^{b}\frac{(b-t)^{n}}{n!}\left\vert f^{(n)}\left( \frac{b-t}{b-a}a+%
\frac{t-a}{b-a}b\right) \right\vert dt.  \notag
\end{eqnarray}%
Since $\left\vert f^{(n)}\right\vert $ is convex on $[a,b],$ we have%
\begin{eqnarray*}
&&\left\vert \int_{a}^{b}f(t)dt-\overset{n-1}{\underset{k=0}{\sum }}\left[ 
\frac{(b-x)^{k+1}+(-1)^{k}(x-a)^{k+1}}{(k+1)!}\right] f^{(k)}(x)\right\vert
\\
&\leq &\frac{1}{n!}\left\{ \int_{a}^{x}(t-a)^{n}\left[ \frac{b-t}{b-a}%
\left\vert f^{(n)}\left( a\right) \right\vert +\frac{t-a}{b-a}\left\vert
f^{(n)}\left( b\right) \right\vert \right] dt\right. \\
&&\text{ \ \ \ \ }\left. +\int_{x}^{b}(b-t)^{n}\left[ \frac{b-t}{b-a}%
\left\vert f^{(n)}\left( a\right) \right\vert +\frac{t-a}{b-a}\left\vert
f^{(n)}\left( b\right) \right\vert \right] dt\right\} .
\end{eqnarray*}%
On the other hand, we have%
\begin{eqnarray*}
\int_{a}^{x}(t-a)^{n}(b-t)dt &=&\frac{(x-a)^{n+1}\left[ (n+2)(b-x)+(x-a)%
\right] }{(n+1)(n+2)}, \\
\int_{a}^{x}(t-a)^{n+1}dt &=&\frac{(x-a)^{n+2}}{(n+2)}, \\
\int_{x}^{b}(b-t)^{n+1}dt &=&\frac{(b-x)^{n+2}}{(n+2)},
\end{eqnarray*}%
and%
\begin{equation*}
\int_{x}^{b}(b-t)^{n}(t-a)dt=\frac{(b-x)^{n+1}\left[ (n+2)(x-a)+(b-x)\right] 
}{(n+1)(n+2)}.
\end{equation*}
This completes the proof.
\end{proof}

\begin{corollary}
With the above assumptions, if we choose $x=\frac{a+b}{2},$ then we get%
\begin{eqnarray*}
&&\left\vert \int_{a}^{b}f(t)dt-\overset{n-1}{\underset{k=0}{\sum }}\left[ 
\frac{1+(-1)^{k}}{(k+1)!}\right] \frac{(b-a)^{k+1}}{2^{k+1}}f^{(k)}\left( 
\frac{a+b}{2}\right) \right\vert \\
&\leq &\frac{(b-a)^{n+1}}{2^{n}(n+1)!}\left[ \frac{\left\vert
f^{(n)}(a)\right\vert +\left\vert f^{(n)}(b)\right\vert }{2}\right] .
\end{eqnarray*}
\end{corollary}

\begin{corollary}
In Theorem \ref{tin}, if we choose $x=a$ and $x=b,$ respectively$,$ we have%
\begin{eqnarray}
&&\left\vert \int_{a}^{b}f(t)dt-\overset{n-1}{\underset{k=0}{\sum }}\frac{%
(b-a)^{k+1}}{(k+1)!}f^{(k)}\left( a\right) \right\vert   \label{4} \\
&\leq &\frac{(b-a)^{n+1}}{(n+2)!}\left[ (n+1)\left\vert
f^{(n)}(a)\right\vert +\left\vert f^{(n)}(b)\right\vert \right]   \notag
\end{eqnarray}
\end{corollary}

\begin{eqnarray}
&&\left\vert \int_{a}^{b}f(t)dt-\overset{n-1}{\underset{k=0}{\sum }}\frac{%
(-1)^{k}(b-a)^{k+1}}{(k+1)!}f^{(k)}\left( b\right) \right\vert   \label{5} \\
&\leq &\frac{(b-a)^{n+1}}{(n+2)!}\left[ \left\vert f^{(n)}(a)\right\vert
+(n+1)\left\vert f^{(n)}(b)\right\vert \right] .  \notag
\end{eqnarray}

\begin{corollary}
Let the conditions of Theorem \ref{tin} hold. Then the following result is
valid. Namely,%
\begin{eqnarray}
&&\left\vert \int_{a}^{b}f(t)dt-\overset{n-1}{\underset{k=0}{\sum }}\frac{%
(b-a)^{k+1}}{(k+1)!}\left[ \frac{f^{(k)}\left( a\right)
+(-1)^{k}f^{(k)}\left( b\right) }{2}\right] \right\vert   \label{6} \\
&\leq &\frac{(b-a)^{n+1}}{(n+1)!}\left[ \frac{\left\vert
f^{(n)}(a)\right\vert +\left\vert f^{(n)}(b)\right\vert }{2}\right] .  \notag
\end{eqnarray}
\end{corollary}

\begin{proof}
Summing the inequalities (\ref{4}) and (\ref{5}) and by using the triangle
inequality, we have the inequality (\ref{6}).
\end{proof}

\begin{corollary}
\label{AA}In Theorem \ref{tin}, if we have $n=1,$ then%
\begin{eqnarray*}
&&\left\vert f(x)-\frac{1}{b-a}\int_{a}^{b}f(t)dt\right\vert \\
&\leq &\frac{1}{(b-a)^{2}}\left\{ \left[ \frac{(x-a)^{2}\left[ 3(b-x)+(x-a)%
\right] }{6}+\frac{(b-x)^{3}}{3}\right] \left\vert f^{\prime }(a)\right\vert
\right. \\
&&\text{ \ \ \ \ \ \ \ \ }\left. +\left[ \frac{(b-x)^{2}\left[ 3(x-a)+(b-x)%
\right] }{6}+\frac{(x-a)^{3}}{3}\right] \left\vert f^{\prime }(b)\right\vert
\right\} .
\end{eqnarray*}
\end{corollary}

\begin{theorem}
\label{yill}Let $f:[a,b]\rightarrow 
%TCIMACRO{\U{211d} }%
%BeginExpansion
\mathbb{R}
%EndExpansion
$ be $n-$time differentiable mapping and $a<b.$ If $f^{(n)}\in L[a,b]$ and $%
\left\vert f^{(n)}\right\vert ^{q}$ is convex on $[a,b],$ then we have the
following inequalities: 
\begin{eqnarray}
&&\left\vert \int_{a}^{b}f(t)dt-\overset{n-1}{\underset{k=0}{\sum }}\left[ 
\frac{(b-x)^{k+1}+(-1)^{k}(x-a)^{k+1}}{(k+1)!}\right] f^{(k)}(x)\right\vert
\label{7} \\
&\leq &\frac{1}{n!(b-a)^{\frac{1}{q}}}\left\{ \frac{(x-a)^{np+1+\frac{1}{q}}%
}{np+1}\left[ \frac{(2b-a-x)}{2}\left\vert f^{(n)}(a)\right\vert ^{q}+\frac{%
(x-a)}{2}\left\vert f^{(n)}(b)\right\vert ^{q}\right] ^{\frac{1}{q}}\right. 
\notag \\
&&\text{ \ \ \ \ \ \ \ \ \ \ \ \ }\left. +\frac{(b-x)^{np+1+\frac{1}{q}}}{%
np+1}\left[ \frac{(b-x)}{2}\left\vert f^{(n)}(a)\right\vert ^{q}+\frac{%
(b+x-2a)}{2}\left\vert f^{(n)}(b)\right\vert ^{q}\right] ^{\frac{1}{q}%
}\right\}  \notag
\end{eqnarray}%
where $\frac{1}{p}+\frac{1}{q}=1.$
\end{theorem}

\begin{proof}
From Lemma \ref{cet}, we have%
\begin{eqnarray*}
&&\left\vert \int_{a}^{b}f(t)dt-\overset{n-1}{\underset{k=0}{\sum }}\left[ 
\frac{(b-x)^{k+1}+(-1)^{k}(x-a)^{k+1}}{(k+1)!}\right] f^{(k)}(x)\right\vert
\\
&\leq &\left\vert \int_{a}^{b}K_{n}(x,t)f^{(n)}(t)dt\right\vert \\
&=&\int_{a}^{x}\frac{(t-a)^{n}}{n!}\left\vert f^{(n)}(t)\right\vert
dt+\int_{x}^{b}\frac{(b-t)^{n}}{n!}\left\vert f^{(n)}(t)\right\vert dt.
\end{eqnarray*}%
By H\"{o}lder inequality, we obtain%
\begin{eqnarray*}
&&\left\vert \int_{a}^{b}f(t)dt-\overset{n-1}{\underset{k=0}{\sum }}\left[ 
\frac{(b-x)^{k+1}+(-1)^{k}(x-a)^{k+1}}{(k+1)!}\right] f^{(k)}(x)\right\vert
\\
&\leq &\frac{1}{n!}\left\{ \left( \int_{a}^{x}(t-a)^{np}dt\right) ^{\frac{1}{%
p}}\left( \int_{a}^{x}\left\vert f^{(n)}(t)\right\vert ^{q}dt\right) ^{\frac{%
1}{q}}\right. \\
&&\text{ \ \ \ }\left. +\left( \int_{x}^{b}(b-t)^{np}dt\right) ^{\frac{1}{p}%
}\left( \int_{x}^{b}\left\vert f^{(n)}(t)\right\vert ^{q}dt\right) ^{\frac{1%
}{q}}\right\} .
\end{eqnarray*}%
Since $\left\vert f^{(n)}\right\vert ^{q}$ is convex on $[a,b]$ and $t=\frac{%
b-t}{b-a}a+\frac{t-a}{b-a}b,$ we have%
\begin{eqnarray*}
&&\left\vert \int_{a}^{b}f(t)dt-\overset{n-1}{\underset{k=0}{\sum }}\left[ 
\frac{(b-x)^{k+1}+(-1)^{k}(x-a)^{k+1}}{(k+1)!}\right] f^{(k)}(x)\right\vert
\\
&\leq &\frac{1}{n!}\left\{ \left( \int_{a}^{x}(t-a)^{np}dt\right) ^{\frac{1}{%
p}}\left( \int_{a}^{x}\left[ \frac{b-t}{b-a}\left\vert f^{(n)}\left(
a\right) \right\vert ^{q}+\frac{t-a}{b-a}\left\vert f^{(n)}\left( b\right)
\right\vert ^{q}\right] dt\right) ^{\frac{1}{q}}\right. \\
&&\text{ \ \ }\left. +\left( \int_{x}^{b}(b-t)^{np}dt\right) ^{\frac{1}{p}%
}\left( \int_{x}^{b}\left[ \frac{b-t}{b-a}\left\vert f^{(n)}\left( a\right)
\right\vert ^{q}+\frac{t-a}{b-a}\left\vert f^{(n)}\left( b\right)
\right\vert ^{q}\right] dt\right) ^{\frac{1}{q}}\right\}
\end{eqnarray*}%
\begin{eqnarray*}
&=&\frac{1}{n!}\left\{ \frac{(x-a)^{np+1}}{np+1}\left[ \frac{(x-a)(2b-a-x)}{%
2(b-a)}\left\vert f^{(n)}(a)\right\vert ^{q}+\frac{(x-a)^{2}}{2(b-a)}%
\left\vert f^{(n)}(b)\right\vert ^{q}\right] ^{\frac{1}{q}}\right. \\
&&\text{\ \ \ \ }\left. +\frac{(b-x)^{np+1}}{np+1}\left[ \frac{(b-x)^{2}}{%
2(b-a)}\left\vert f^{(n)}(a)\right\vert ^{q}+\frac{(b-x)(b+x-2a)}{2(b-a)}%
\left\vert f^{(n)}(b)\right\vert ^{q}\right] ^{\frac{1}{q}}\right\} .
\end{eqnarray*}%
This completes the proof.
\end{proof}

\begin{corollary}
Assume that $f$ is as in Teorem \ref{yill}. If we choose $x=\frac{a+b}{2},$
then we have%
\begin{eqnarray*}
&&\left\vert \int_{a}^{b}f(t)dt-\overset{n-1}{\underset{k=0}{\sum }}\left[ 
\frac{1+(-1)^{k}}{(k+1)!}\right] \left( \frac{b-a}{2}\right)
^{k+1}f^{(k)}\left( \frac{a+b}{2}\right) \right\vert \\
&\leq &\left( \frac{b-a}{2}\right) ^{np+1+\frac{1}{q}}\frac{1}{(np+1)n!} \\
&&\times \left\{ \left[ \frac{3\left\vert f^{(n)}(a)\right\vert
^{q}+\left\vert f^{(n)}(b)\right\vert ^{q}}{4}\right] ^{\frac{1}{q}}+\left[ 
\frac{\left\vert f^{(n)}(a)\right\vert ^{q}+3\left\vert
f^{(n)}(b)\right\vert ^{q}}{4}\right] ^{\frac{1}{q}}\right\} .
\end{eqnarray*}
\end{corollary}

\begin{corollary}
With the above assumptions, if we choose $x=a$ and $x=b,$ respectively$,$ we
have%
\begin{eqnarray}
&&\left\vert \int_{a}^{b}f(t)dt-\overset{n-1}{\underset{k=0}{\sum }}\frac{%
(b-a)^{k+1}}{(k+1)!}f^{(k)}\left( a\right) \right\vert   \label{8} \\
&\leq &\frac{(b-a)^{np+1+\frac{1}{q}}}{(np+1)n!}\left[ \frac{\left\vert
f^{(n)}(a)\right\vert ^{q}+\left\vert f^{(n)}(b)\right\vert ^{q}}{2}\right]
^{\frac{1}{q}}  \notag
\end{eqnarray}
\end{corollary}

\begin{eqnarray}
&&\left\vert \int_{a}^{b}f(t)dt-\overset{n-1}{\underset{k=0}{\sum }}\frac{%
(-1)^{k}(b-a)^{k+1}}{(k+1)!}f^{(k)}\left( b\right) \right\vert   \label{9} \\
&\leq &\frac{(b-a)^{np+1+\frac{1}{q}}}{(np+1)n!}\left[ \frac{\left\vert
f^{(n)}(a)\right\vert ^{q}+\left\vert f^{(n)}(b)\right\vert ^{q}}{2}\right]
^{\frac{1}{q}}.  \notag
\end{eqnarray}

\begin{corollary}
Let the conditions of Theorem \ref{yill} hold. Then the following result is
valid. Namely,%
\begin{eqnarray}
&&\left\vert \int_{a}^{b}f(t)dt-\overset{n-1}{\underset{k=0}{\sum }}\frac{%
(b-a)^{k+1}}{(k+1)!}\left[ \frac{f^{(k)}\left( a\right)
+(-1)^{k}f^{(k)}\left( b\right) }{2}\right] \right\vert   \label{10} \\
&\leq &\frac{(b-a)^{n+1}}{(n+1)!}\left[ \frac{\left\vert
f^{(n)}(a)\right\vert +\left\vert f^{(n)}(b)\right\vert }{2}\right] .  \notag
\end{eqnarray}
\end{corollary}

\begin{proof}
Summing the inequalities (\ref{8}) and (\ref{9}) and by using the triangle
inequality, we have the inequality (\ref{10}).
\end{proof}

\begin{corollary}
In the inequalities (\ref{7}), if we choose $n=1,$ then we have%
\begin{eqnarray*}
&&\left\vert f(x)-\frac{1}{b-a}\int_{a}^{b}f(t)dt\right\vert \\
&\leq &\frac{1}{(b-a)^{1+\frac{1}{q}}}\left\{ \frac{(x-a)^{p+1+\frac{1}{q}}}{%
p+1}\left[ \frac{(2b-a-x)}{2}\left\vert f^{\prime }(a)\right\vert ^{q}+\frac{%
(x-a)}{2}\left\vert f^{\prime }(b)\right\vert ^{q}\right] ^{\frac{1}{q}%
}\right. \\
&&\text{ \ \ \ \ \ \ \ \ \ \ \ \ \ }\left. +\frac{(b-x)^{p+1+\frac{1}{q}}}{%
p+1}\left[ \frac{(b-x)}{2}\left\vert f^{\prime }(a)\right\vert ^{q}+\frac{%
(b+x-2a)}{2}\left\vert f^{\prime }(b)\right\vert ^{q}\right] ^{\frac{1}{q}%
}\right\} .
\end{eqnarray*}
\end{corollary}

\begin{theorem}
\label{A}Let $f:[a,b]\rightarrow 
%TCIMACRO{\U{211d} }%
%BeginExpansion
\mathbb{R}
%EndExpansion
$ be $n-$time differentiable mapping and $a<b.$ If $f^{(n)}\in L[a,b]$ and $%
\left\vert f^{(n)}\right\vert ^{q}$ is convex on $[a,b],$ $\frac{1}{p}+\frac{%
1}{q}=1,$ then we have 
\begin{eqnarray}
&&  \label{11} \\
&&\left\vert \int_{a}^{b}f(t)dt-\overset{n-1}{\underset{k=0}{\sum }}\left[ 
\frac{(b-x)^{k+1}+(-1)^{k}(x-a)^{k+1}}{(k+1)!}\right] f^{(k)}(x)\right\vert 
\notag \\
&\leq &\frac{1}{n!(b-a)^{\frac{1}{q}}(p+2)^{\frac{1}{q}}}\left( \frac{q-1}{%
nq+q-p-1}\right) ^{1-\frac{1}{q}}  \notag \\
&&\times \left\{ (x-a)^{n+1}\left[ \frac{(p+2)(b-x)+(x-a)}{(p+1)}\left\vert
f^{(n)}(a)\right\vert ^{q}+(x-a)^{p+1}\left\vert f^{(n)}(b)\right\vert ^{q}%
\right] ^{\frac{1}{q}}\right.  \notag \\
&&\text{ \ \ }\left. +(b-x)^{n+1}\left[ (b-x)^{p+1}\left\vert
f^{(n)}(a)\right\vert ^{q}+\frac{(p+2)(x-a)+(b-x)}{(p+1)}\left\vert
f^{(n)}(b)\right\vert ^{q}\right] ^{\frac{1}{q}}\right\} .  \notag
\end{eqnarray}
\end{theorem}

\begin{proof}
From Lemma \ref{cet} and using the properties of modulus, we have%
\begin{eqnarray*}
&&\left\vert \int_{a}^{b}f(t)dt-\overset{n-1}{\underset{k=0}{\sum }}\left[ 
\frac{(b-x)^{k+1}+(-1)^{k}(x-a)^{k+1}}{(k+1)!}\right] f^{(k)}(x)\right\vert
\\
&\leq &\left\vert \int_{a}^{b}K_{n}(x,t)f^{(n)}(t)dt\right\vert \\
&=&\int_{a}^{x}\frac{(t-a)^{n}}{n!}\left\vert f^{(n)}(t)\right\vert
dt+\int_{x}^{b}\frac{(b-t)^{n}}{n!}\left\vert f^{(n)}(t)\right\vert dt \\
&=&\frac{1}{n!}\left\{ \int_{a}^{x}(t-a)^{n}\left\vert f^{(n)}(t)\right\vert
dt+\int_{x}^{b}(b-t)^{n}\left\vert f^{(n)}(t)\right\vert dt\right\} \\
&=&\frac{1}{n!}\left\{ \int_{a}^{x}\frac{(t-a)^{n}(t-a)^{\frac{p}{q}}}{%
(t-a)^{\frac{p}{q}}}\left\vert f^{(n)}(t)\right\vert dt+\int_{x}^{b}\frac{%
(b-t)^{n}(b-t)^{\frac{p}{q}}}{(b-t)^{\frac{p}{q}}}\left\vert
f^{(n)}(t)\right\vert dt\right\}
\end{eqnarray*}%
By H\"{o}lder inequality, we obtain%
\begin{eqnarray*}
&&\left\vert \int_{a}^{b}f(t)dt-\overset{n-1}{\underset{k=0}{\sum }}\left[ 
\frac{(b-x)^{k+1}+(-1)^{k}(x-a)^{k+1}}{(k+1)!}\right] f^{(k)}(x)\right\vert
\\
&\leq &\frac{1}{n!}\left\{ \left( \int_{a}^{x}\left[ \frac{(t-a)^{n}}{(t-a)^{%
\frac{p}{q}}}\right] ^{\frac{q}{q-1}}dt\right) ^{1-\frac{1}{q}}\left(
\int_{a}^{x}(t-a)^{p}\left\vert f^{(n)}(t)\right\vert ^{q}dt\right) ^{\frac{1%
}{q}}\right. \\
&&\text{ \ \ \ }\left. +\left( \int_{x}^{b}\left[ \frac{(b-t)^{n}}{(b-t)^{%
\frac{p}{q}}}\right] ^{\frac{q}{q-1}}dt\right) ^{1-\frac{1}{q}}\left(
\int_{x}^{b}(b-t)^{p}\left\vert f^{(n)}(t)\right\vert ^{q}dt\right) ^{\frac{1%
}{q}}\right\} .
\end{eqnarray*}%
Since $\left\vert f^{(n)}\right\vert ^{q}$ is convex on $[a,b]$ and $t=\frac{%
b-t}{b-a}a+\frac{t-a}{b-a}b,$ we have%
\begin{eqnarray*}
&&\left\vert \int_{a}^{b}f(t)dt-\overset{n-1}{\underset{k=0}{\sum }}\left[ 
\frac{(b-x)^{k+1}+(-1)^{k}(x-a)^{k+1}}{(k+1)!}\right] f^{(k)}(x)\right\vert
\\
&\leq &\frac{1}{n!}\left\{ \left( \int_{a}^{x}(t-a)^{\frac{nq-p}{q-1}%
}dt\right) ^{1-\frac{1}{q}}\left( \int_{a}^{x}(t-a)^{p}\left[ \frac{b-t}{b-a}%
\left\vert f^{(n)}\left( a\right) \right\vert ^{q}+\frac{t-a}{b-a}\left\vert
f^{(n)}\left( b\right) \right\vert ^{q}\right] dt\right) ^{\frac{1}{q}%
}\right. \\
&&\text{ \ \ }\left. +\left( \int_{x}^{b}(b-t)^{\frac{nq-p}{q-1}}dt\right)
^{1-\frac{1}{q}}\left( \int_{x}^{b}(b-t)^{p}\left[ \frac{b-t}{b-a}\left\vert
f^{(n)}\left( a\right) \right\vert ^{q}+\frac{t-a}{b-a}\left\vert
f^{(n)}\left( b\right) \right\vert ^{q}\right] dt\right) ^{\frac{1}{q}%
}\right\} \\
&=&\frac{1}{n!(b-a)^{\frac{1}{q}}(p+2)^{\frac{1}{q}}}\left( \frac{q-1}{%
nq+q-p-1}\right) ^{1-\frac{1}{q}} \\
&&\times \left\{ (x-a)^{n+1}\left[ \frac{(p+2)(b-x)+(x-a)}{(p+1)}\left\vert
f^{(n)}(a)\right\vert ^{q}+(x-a)^{p+1}\left\vert f^{(n)}(b)\right\vert ^{q}%
\right] ^{\frac{1}{q}}\right. \\
&&\text{\ \ }\left. +(b-x)^{n+1}\left[ (b-x)^{p+1}\left\vert
f^{(n)}(a)\right\vert ^{q}+\frac{(p+2)(x-a)+(b-x)}{(p+1)}\left\vert
f^{(n)}(b)\right\vert ^{q}\right] ^{\frac{1}{q}}\right\} .
\end{eqnarray*}%
By using the fact that%
\begin{eqnarray*}
\int_{a}^{x}(t-a)^{\frac{nq-p}{q-1}}dt &=&\frac{q-1}{nq+q-p-1}(x-a)^{\frac{%
nq+q-p-1}{q-1}}, \\
\int_{x}^{b}(b-t)^{\frac{nq-p}{q-1}}dt &=&\frac{q-1}{nq+q-p-1}(b-x)^{\frac{%
nq+q-p-1}{q-1}}
\end{eqnarray*}%
we get the inequality (\ref{11}), which completes the proof of the theorem.
\end{proof}

\begin{corollary}
Assume that $f$ is as in Teorem \ref{A}. If we choose $x=\frac{a+b}{2},$
then we have%
\begin{eqnarray*}
&&\left\vert \int_{a}^{b}f(t)dt-\overset{n-1}{\underset{k=0}{\sum }}\left[ 
\frac{1+(-1)^{k}}{(k+1)!}\right] \left( \frac{b-a}{2}\right)
^{k+1}f^{(k)}\left( \frac{a+b}{2}\right) \right\vert \\
&\leq &\frac{\left( b-a\right) ^{n+1}}{n!2^{n+1+\frac{1}{q}}(p+2)^{\frac{1}{q%
}}}\left( \frac{q-1}{nq+q-p-1}\right) ^{1-\frac{1}{q}} \\
&&\times \left\{ \left[ \frac{p+3}{p+1}\left\vert f^{(n)}(a)\right\vert
^{q}+\left( \frac{b-a}{2}\right) ^{p}\left\vert f^{(n)}(b)\right\vert ^{q}%
\right] ^{\frac{1}{q}}\right. \\
&&\text{ \ }\left. +\left[ \left( \frac{b-a}{2}\right) ^{p}\left\vert
f^{(n)}(a)\right\vert ^{q}+\frac{p+3}{p+1}\left\vert f^{(n)}(b)\right\vert
^{q}\right] ^{\frac{1}{q}}\right\} .
\end{eqnarray*}
\end{corollary}

\begin{corollary}
With the above assumptions, if we choose $x=a$ and $x=b,$ respectively$,$ we
have%
\begin{eqnarray}
&&\left\vert \int_{a}^{b}f(t)dt-\overset{n-1}{\underset{k=0}{\sum }}\frac{%
(b-a)^{k+1}}{(k+1)!}f^{(k)}\left( a\right) \right\vert   \label{12} \\
&\leq &\frac{(b-a)^{n+1}}{n!(p+2)^{\frac{1}{q}}}\left( \frac{q-1}{nq+q-p-1}%
\right) ^{1-\frac{1}{q}}  \notag \\
&&\times \left[ (b-a)^{p}\left\vert f^{(n)}(a)\right\vert ^{q}+\frac{1}{p+1}%
\left\vert f^{(n)}(b)\right\vert ^{q}\right] ^{\frac{1}{q}}  \notag
\end{eqnarray}
\end{corollary}

\begin{eqnarray}
&&\left\vert \int_{a}^{b}f(t)dt-\overset{n-1}{\underset{k=0}{\sum }}\frac{%
(-1)^{k}(b-a)^{k+1}}{(k+1)!}f^{(k)}\left( b\right) \right\vert   \label{13}
\\
&\leq &\frac{(b-a)^{n+1}}{n!(p+2)^{\frac{1}{q}}}\left( \frac{q-1}{nq+q-p-1}%
\right) ^{1-\frac{1}{q}}  \notag \\
&&\times \left[ \frac{1}{p+1}\left\vert f^{(n)}(a)\right\vert
^{q}+(b-a)^{p}\left\vert f^{(n)}(b)\right\vert ^{q}\right] ^{\frac{1}{q}}. 
\notag
\end{eqnarray}

\begin{corollary}
Let the conditions of Theorem \ref{A} hold. Then the following result is
valid. Namely,%
\begin{eqnarray}
&&\left\vert \int_{a}^{b}f(t)dt-\overset{n-1}{\underset{k=0}{\sum }}\frac{%
(b-a)^{k+1}}{(k+1)!}\left[ \frac{f^{(k)}\left( a\right)
+(-1)^{k}f^{(k)}\left( b\right) }{2}\right] \right\vert   \label{14} \\
&\leq &\frac{(b-a)^{n+1}}{n!(p+2)^{\frac{1}{q}}}\left( \frac{q-1}{nq+q-p-1}%
\right) ^{1-\frac{1}{q}}  \notag \\
&&\times \left\{ \left[ (b-a)^{p}\left\vert f^{(n)}(a)\right\vert ^{q}+\frac{%
1}{p+1}\left\vert f^{(n)}(b)\right\vert ^{q}\right] ^{\frac{1}{q}}\right.  
\notag \\
&&\text{ \ \ }\left. +\left[ \frac{1}{p+1}\left\vert f^{(n)}(a)\right\vert
^{q}+(b-a)^{p}\left\vert f^{(n)}(b)\right\vert ^{q}\right] ^{\frac{1}{q}%
}\right\} .  \notag
\end{eqnarray}
\end{corollary}

\begin{proof}
Summing the inequalities (\ref{12}) and (\ref{13}) and by using the triangle
inequality, we have the inequality (\ref{14}).
\end{proof}

\begin{corollary}
In the inequalities (\ref{11}), if we choose $n=1,$ then we have%
\begin{eqnarray*}
&&\left\vert f(x)-\frac{1}{b-a}\int_{a}^{b}f(t)dt\right\vert \\
&\leq &\frac{1}{(b-a)^{\frac{1}{q}}(p+2)^{\frac{1}{q}}}\left( \frac{q-1}{%
2q-p-1}\right) ^{1-\frac{1}{q}} \\
&&\times \left\{ (x-a)^{2}\left[ \frac{(p+2)(b-x)+(x-a)}{(p+1)}\left\vert
f^{\prime }(a)\right\vert ^{q}+(x-a)^{p+1}\left\vert f^{\prime
}(b)\right\vert ^{q}\right] ^{\frac{1}{q}}\right. \\
&&\text{ \ \ }\left. +(b-x)^{2}\left[ (b-x)^{p+1}\left\vert f^{\prime
}(a)\right\vert ^{q}+\frac{(p+2)(x-a)+(b-x)}{(p+1)}\left\vert f^{\prime
}(b)\right\vert ^{q}\right] ^{\frac{1}{q}}\right\} .
\end{eqnarray*}
\end{corollary}

\begin{theorem}
\label{B}For $n\geq 1,$ let $f:[a,b]\rightarrow 
%TCIMACRO{\U{211d} }%
%BeginExpansion
\mathbb{R}
%EndExpansion
$ be $n-$time differentiable mapping and $a<b.$ If $f^{(n)}\in L[a,b]$ and $%
\left\vert f^{(n)}\right\vert ^{q}$ is convex on $[a,b]$ and $q\geq 1,$ then
we have the following inequality: 
\begin{eqnarray}
&&  \label{15} \\
&&\left\vert \int_{a}^{b}f(t)dt-\overset{n-1}{\underset{k=0}{\sum }}\left[ 
\frac{(b-x)^{k+1}+(-1)^{k}(x-a)^{k+1}}{(k+1)!}\right] f^{(k)}(x)\right\vert 
\notag \\
&\leq &\frac{1}{(n+1)!(b-a)^{\frac{1}{q}}(n+2)^{\frac{1}{q}}}  \notag \\
&&\times \left\{ (x-a)^{n+1}\left[ \left[ (n+2)(b-x)+(x-a)\right] \left\vert
f^{(n)}(a)\right\vert ^{q}+(n+1)(x-a)\left\vert f^{(n)}(b)\right\vert ^{q}%
\right] ^{\frac{1}{q}}\right.  \notag \\
&&\text{ \ \ }\left. +(b-x)^{n+1}\left[ (n+1)(b-x)\left\vert
f^{(n)}(a)\right\vert ^{q}+\left[ (n+2)(x-a)+(b-x)\right] \left\vert
f^{(n)}(b)\right\vert ^{q}\right] ^{\frac{1}{q}}\right\} .  \notag
\end{eqnarray}
\end{theorem}

\begin{proof}
From Lemma \ref{cet} and using the properties of modulus, we obtain%
\begin{eqnarray*}
&&\left\vert \int_{a}^{b}f(t)dt-\overset{n-1}{\underset{k=0}{\sum }}\left[ 
\frac{(b-x)^{k+1}+(-1)^{k}(x-a)^{k+1}}{(k+1)!}\right] f^{(k)}(x)\right\vert
\\
&\leq &\left\vert \int_{a}^{b}K_{n}(x,t)f^{(n)}(t)dt\right\vert \\
&=&\frac{1}{n!}\left\{ \int_{a}^{x}(t-a)^{n}\left\vert f^{(n)}(t)\right\vert
dt+\int_{x}^{b}(b-t)^{n}\left\vert f^{(n)}(t)\right\vert dt\right\} .
\end{eqnarray*}%
By Power-mean inequality, we obtain%
\begin{eqnarray*}
&&\left\vert \int_{a}^{b}f(t)dt-\overset{n-1}{\underset{k=0}{\sum }}\left[ 
\frac{(b-x)^{k+1}+(-1)^{k}(x-a)^{k+1}}{(k+1)!}\right] f^{(k)}(x)\right\vert
\\
&\leq &\frac{1}{n!}\left\{ \left( \int_{a}^{x}(t-a)^{n}dt\right) ^{1-\frac{1%
}{q}}\left( \int_{a}^{x}(t-a)^{n}\left\vert f^{(n)}(t)\right\vert
^{q}dt\right) ^{\frac{1}{q}}\right. \\
&&\text{ \ \ \ }\left. +\left( \int_{x}^{b}(b-t)^{n}dt\right) ^{1-\frac{1}{q}%
}\left( \int_{x}^{b}(b-t)^{n}\left\vert f^{(n)}(t)\right\vert ^{q}dt\right)
^{\frac{1}{q}}\right\} .
\end{eqnarray*}%
Since $\left\vert f^{(n)}\right\vert ^{q}$ is convex on $[a,b]$ and $t=\frac{%
b-t}{b-a}a+\frac{t-a}{b-a}b,$ we have%
\begin{eqnarray*}
&&\left\vert \int_{a}^{b}f(t)dt-\overset{n-1}{\underset{k=0}{\sum }}\left[ 
\frac{(b-x)^{k+1}+(-1)^{k}(x-a)^{k+1}}{(k+1)!}\right] f^{(k)}(x)\right\vert
\\
&\leq &\frac{1}{n!}\left\{ \left( \frac{(x-a)^{n+1}}{n+1}\right) ^{1-\frac{1%
}{q}}\left( \int_{a}^{x}(t-a)^{n}\left[ \frac{b-t}{b-a}\left\vert
f^{(n)}\left( a\right) \right\vert ^{q}+\frac{t-a}{b-a}\left\vert
f^{(n)}\left( b\right) \right\vert ^{q}\right] dt\right) ^{\frac{1}{q}%
}\right. \\
&&\text{ \ \ }\left. +\left( \frac{(b-x)^{n+1}}{n+1}\right) ^{1-\frac{1}{q}%
}\left( \int_{x}^{b}(b-t)^{n}\left[ \frac{b-t}{b-a}\left\vert f^{(n)}\left(
a\right) \right\vert ^{q}+\frac{t-a}{b-a}\left\vert f^{(n)}\left( b\right)
\right\vert ^{q}\right] dt\right) ^{\frac{1}{q}}\right\} \\
&=&\frac{1}{(n+1)!(b-a)^{\frac{1}{q}}(n+2)^{\frac{1}{q}}} \\
&&\times \left\{ (x-a)^{n+1}\left[ \left[ (n+2)(b-x)+(x-a)\right] \left\vert
f^{(n)}(a)\right\vert ^{q}+(n+1)(x-a)\left\vert f^{(n)}(b)\right\vert ^{q}%
\right] ^{\frac{1}{q}}\right. \\
&&\text{\ }\left. +(b-x)^{n+1}\left[ (n+1)(b-x)\left\vert
f^{(n)}(a)\right\vert ^{q}+\left[ (n+2)(x-a)+(b-x)\right] \left\vert
f^{(n)}(b)\right\vert ^{q}\right] ^{\frac{1}{q}}\right\} .
\end{eqnarray*}%
Hence the proof of the theorem is completed.
\end{proof}

\begin{corollary}
With the above assumptions, if we choose $x=\frac{a+b}{2},$ then we have%
\begin{eqnarray*}
&&\left\vert \int_{a}^{b}f(t)dt-\overset{n-1}{\underset{k=0}{\sum }}\left[ 
\frac{1+(-1)^{k}}{(k+1)!}\right] \left( \frac{b-a}{2}\right)
^{k+1}f^{(k)}\left( \frac{a+b}{2}\right) \right\vert \\
&\leq &\frac{\left( b-a\right) ^{n+1}}{(n+1)!2^{n+1+\frac{1}{q}}(n+2)^{\frac{%
1}{q}}} \\
&&\times \left\{ \left[ (n+3)\left\vert f^{(n)}(a)\right\vert
^{q}+(n+1)\left\vert f^{(n)}(b)\right\vert ^{q}\right] ^{\frac{1}{q}}\right.
\\
&&\text{ \ }\left. +\left[ (n+1)\left\vert f^{(n)}(a)\right\vert
^{q}+(n+3)\left\vert f^{(n)}(b)\right\vert ^{q}\right] ^{\frac{1}{q}%
}\right\} .
\end{eqnarray*}
\end{corollary}

\begin{corollary}
In Theorem \ref{B}, if we choose $x=a$ and $x=b,$ respectively$,$ we have%
\begin{eqnarray}
&&\left\vert \int_{a}^{b}f(t)dt-\overset{n-1}{\underset{k=0}{\sum }}\frac{%
(b-a)^{k+1}}{(k+1)!}f^{(k)}\left( a\right) \right\vert   \label{16} \\
&\leq &\frac{(b-a)^{n+1}}{(n+1)!(n+2)^{\frac{1}{q}}}\left[ (n+1)\left\vert
f^{(n)}(a)\right\vert ^{q}+\left\vert f^{(n)}(b)\right\vert ^{q}\right] ^{%
\frac{1}{q}}  \notag
\end{eqnarray}
\end{corollary}

\begin{eqnarray}
&&\left\vert \int_{a}^{b}f(t)dt-\overset{n-1}{\underset{k=0}{\sum }}\frac{%
(-1)^{k}(b-a)^{k+1}}{(k+1)!}f^{(k)}\left( b\right) \right\vert   \label{17}
\\
&\leq &\frac{(b-a)^{n+1}}{(n+1)!(n+2)^{\frac{1}{q}}}\left[ \left\vert
f^{(n)}(a)\right\vert ^{q}+(n+1)\left\vert f^{(n)}(b)\right\vert ^{q}\right]
^{\frac{1}{q}}.  \notag
\end{eqnarray}

\begin{corollary}
Let the conditions of Theorem \ref{B} hold. Then the following result is
valid:%
\begin{eqnarray}
&&\left\vert \int_{a}^{b}f(t)dt-\overset{n-1}{\underset{k=0}{\sum }}\frac{%
(b-a)^{k+1}}{(k+1)!}\left[ \frac{f^{(k)}\left( a\right)
+(-1)^{k}f^{(k)}\left( b\right) }{2}\right] \right\vert   \label{18} \\
&\leq &\frac{(b-a)^{n+1}}{2(n+1)!(n+2)^{\frac{1}{q}}}  \notag \\
&&\times \left\{ \left[ (n+1)\left\vert f^{(n)}(a)\right\vert
^{q}+\left\vert f^{(n)}(b)\right\vert ^{q}\right] ^{\frac{1}{q}}\right.  
\notag \\
&&\text{\ \ }\left. +\left[ \left\vert f^{(n)}(a)\right\vert
^{q}+(n+1)\left\vert f^{(n)}(b)\right\vert ^{q}\right] ^{\frac{1}{q}%
}\right\} .  \notag
\end{eqnarray}
\end{corollary}

\begin{proof}
Summing the inequalities (\ref{16}) and (\ref{17}) and by using the triangle
inequality, we have the inequality (\ref{18}).
\end{proof}

\begin{corollary}
\label{AB}In the inequalities (\ref{11}), if we choose $n=1,$ then we have%
\begin{eqnarray*}
&&\left\vert f(x)-\frac{1}{b-a}\int_{a}^{b}f(t)dt\right\vert \\
&\leq &\frac{1}{2(b-a)^{\frac{1}{q}}}\left\{ (x-a)^{2}\left[ \frac{(3b-2x-a)%
}{3}\left\vert f^{\prime }(a)\right\vert ^{q}+\frac{2(x-a)}{3}\left\vert
f^{\prime }(b)\right\vert ^{q}\right] ^{\frac{1}{q}}\right. \\
&&\text{ \ \ \ \ \ \ \ \ \ \ \ \ }\left. +(b-x)^{2}\left[ \frac{2(b-x)}{3}%
\left\vert f^{\prime }(a)\right\vert ^{q}+\frac{(b+2x-3a)}{3}\left\vert
f^{\prime }(b)\right\vert ^{q}\right] ^{\frac{1}{q}}\right\} .
\end{eqnarray*}
\end{corollary}

\section{APPLICATIONS TO SPECIAL MEANS}

We now consider the means for arbitrary real numbers $\alpha ,\beta $ $%
(\alpha \neq \beta ).$ We take

\begin{enumerate}
\item $Arithmetic$ $mean:$%
\begin{equation*}
A(\alpha ,\beta )=\frac{\alpha +\beta }{2},\text{ \ }\alpha ,\beta \in 
%TCIMACRO{\U{211d} }%
%BeginExpansion
\mathbb{R}
%EndExpansion
^{+}.
\end{equation*}

\item $Logarithmic$ $mean$:%
\begin{equation*}
L(\alpha ,\beta )=\frac{\alpha -\beta }{\ln \left\vert \alpha \right\vert
-\ln \left\vert \beta \right\vert },\text{ \ \ }\left\vert \alpha
\right\vert \neq \left\vert \beta \right\vert ,\text{ }\alpha ,\beta \neq 0,%
\text{ }\alpha ,\beta \in 
%TCIMACRO{\U{211d} }%
%BeginExpansion
\mathbb{R}
%EndExpansion
^{+}.
\end{equation*}
\end{enumerate}

Now using the results of Section 2, we give some applications for special
means of real numbers.

\begin{proposition}
Let $a,b\in 
%TCIMACRO{\U{211d} }%
%BeginExpansion
\mathbb{R}
%EndExpansion
,$ $0<a<b$ and $n\in 
%TCIMACRO{\U{2124} }%
%BeginExpansion
\mathbb{Z}
%EndExpansion
,$ $\left\vert n\right\vert \geq 1,$ then, the following inequality holds:%
\begin{eqnarray*}
\left\vert L_{n}^{n}(a,b)-x^{n}\right\vert &\leq &\frac{\left\vert
n\right\vert }{(b-a)^{2}}\left\{ \frac{\left[ (x-a)^{2}(3b-a-2x)+2(b-x)^{3}%
\right] .a^{n-1}}{6}\right. \\
&&\text{ \ \ \ \ \ \ \ \ \ \ }\left. +\frac{\left[
(b-x)^{2}(b-3a+2x)+2(x-a)^{3}\right] .b^{n-1}}{6}\right\} .
\end{eqnarray*}
\end{proposition}

\begin{proof}
The proof is obvious from Corollary \ref{AA} applied to the convex mapping $%
f(x)=x^{n},$ $x\in \lbrack a,b],$ $n\in 
%TCIMACRO{\U{2124} }%
%BeginExpansion
\mathbb{Z}
%EndExpansion
$.
\end{proof}

\begin{proposition}
Let $a,b\in 
%TCIMACRO{\U{211d} }%
%BeginExpansion
\mathbb{R}
%EndExpansion
,$ $0<a<b$ and $n\in 
%TCIMACRO{\U{2124} }%
%BeginExpansion
\mathbb{Z}
%EndExpansion
,$ $\left\vert n\right\vert \geq 1,$ then, for all $q\geq 1,$ the following
inequality holds:%
\begin{eqnarray*}
\left\vert L_{n}^{n}(a,b)-x^{n}\right\vert &\leq &\frac{\left\vert
n\right\vert }{2(b-a)^{\frac{1}{q}}}\left\{ (x-a)^{2}\left[ \frac{%
(3b-2x-a)\left( a^{n-1}\right) ^{q}+2(x-a)\left( b^{n-1}\right) ^{q}}{3}%
\right] ^{\frac{1}{q}}\right. \\
&&\text{ \ \ \ \ \ \ \ \ \ \ }\left. +(b-x)^{2}\left[ \frac{2(b-x)\left(
a^{n-1}\right) ^{q}+(b+2x-3a)\left( b^{n-1}\right) ^{q}}{3}\right] ^{\frac{1%
}{q}}\right\} .
\end{eqnarray*}
\end{proposition}

\begin{proof}
The proof is obvious from Corollary \ref{AB} applied to the convex mapping $%
f(x)=x^{n},$ $x\in \lbrack a,b],$ $n\in 
%TCIMACRO{\U{2124} }%
%BeginExpansion
\mathbb{Z}
%EndExpansion
$.
\end{proof}

\end{document}